\documentclass[11pt]{amsart}
\usepackage[latin1]{inputenc}
\usepackage[english]{babel}
\usepackage{amsmath}
\usepackage{amsfonts}
\usepackage{amssymb}
\usepackage{amsthm}
\usepackage{scrpage2}
\usepackage{enumerate}
\usepackage{textcomp}
\usepackage[dvipdf]{graphicx}%

\usepackage{hyperref} 
\swapnumbers
\newtheorem{thm}{Theorem}[section]
\newtheorem{prop}[thm]{Proposition}
\newtheorem{cor}[thm]{Corollary}

\newtheorem{lem}[thm]{Lemma}
\newtheorem{qu}[thm]{Question}
	
\newtheorem{fact}[thm]{Fact}

\theoremstyle{definition}
\newtheorem{de}[thm]{Definition}

\newtheorem{ex}[thm]{Example}

\newtheorem{rem}[thm]{Remark}

\newtheorem{con}[thm]{Conjecture}

\renewenvironment{proof}{\par\noindent {\em Proof: }}{\hfill$\Box$\medskip}
\theoremstyle{plain}

\newcommand{\cO}{\ensuremath{\mathcal{O}} }

\newcommand{\cB}{\ensuremath{\mathcal{B}} }
\newcommand{\cM}{\ensuremath{\mathcal{M}} }

\newcommand{\cN}{\ensuremath{\mathcal{N}} }

\newcommand{\cS}{\ensuremath{\mathcal{S}} }
\newcommand{\cT}{\ensuremath{\mathcal{T}} }

\newcommand{\cL}{\ensuremath{\mathcal{L}} }

\newcommand{\ch}{\ensuremath{\textrm{char}} }

\newcommand{\N}{\ensuremath{\mathbb{N}} }

\automark[subsection]{section}

\author[K. Dupont]{Katharina Dupont$^*$}
\thanks{$^*$ Partially supported by a Minerva fellowship and a post-doctoral fellowship from Ben Gurion University}
\address{$^*$FB Mathematik und Statistik\\
Universit\"at Konstanz\\
78457 Konstanz\\
Germany} \email{katharina.dupont@uni-konstanz.de}
\author[A. Hasson]{Assaf Hasson$^{**}$}
\thanks{$^{**}$ Partially supported by GIF grant 2165/2011 and by ISF grant  181/16}
\address{$^{**}$Department of mathematics\\
Ben Gurion University of the Negev\\
Be'er Sehva\\
Israel} \email{hassonas@math.bgu.ac.il}

\author[S. Kuhlmann]{Salma Kuhlmann$^{***}$}
\thanks{$^{***}$ Partially supported by AFF grant from the University of Konstanz}
\address{$^{***}$FB Mathematik und Statistik\\
Universit\"at Konstanz\\
78457 Konstanz\\
Germany} \email{salma.kuhlmann@uni-konstanz.de}

\title[Definable valuations in NIP fields]{Definable Valuations induced by multiplicative subgroups and NIP Fields} 


\date{\today}

\begin{document}
\maketitle
\begin{abstract}
We study the algebraic implications of the non-independence property (NIP)
and variants thereof (dp-minimality) on infinite fields, motivated by the
conjecture that all such fields which are neither real closed nor separably
closed admit a (definable) henselian valuation. Our results mainly focus on Hahn
fields and build up on Will Johnson's preprint "dp-minimal fields", arXiv:
1507.02745v1, July 2015.
\end{abstract}

\section*{Introduction}

The classification of $\omega$-stable fields \cite[Theorem 3.1]{PoiGroups} and later of super-stable fields \cite{ChSh} is a cornerstone in the development of the interactions between model theory, algebra and geometry. Ever since, the classification of algebraic structures according to their model theoretic properties is a recurring theme in model theory. Despite some success in the  classification of groups of finite rank (with respect to various notions of rank), e.g..  \cite{EaKrPi},\cite[Section 4]{WaBook} (essentially, generalising results from the stable context), and most notably in the o-minimal setting (e.g., \cite{HrPePi} and many references therein) little progress has been made in the classification of infinite stable (let alone simple) fields. Indeed, most experts view the conjecture asserting that (super) simple fields are bounded (perfect) PAC, and even the considerably weaker conjecture that stable fields are separably closed to be out of the reach of existing techniques.

In the last decade or so the increasing interest in theories without the independence property (NIP theories), associated usually with the solution of Pillay's conjecture \cite{HrPePi} and with the study of algebraically closed valued fields, led naturally to analogous classification problems in that context. In its full generality, the problem of classifying NIP fields encompasses the classification of stable fields, and may be too ambitious. In \cite{Sh863}, as an attempt to find the right analogue of super-stability in the context of NIP theories, Shelah introduced the notion of \emph{strong NIP}. As part of establishing this analogy, Shelah showed \cite[Claim 5.40]{Sh863} that the theory of a separably closed field  that is not algebraically closed is not strongly NIP. In fact Shelah's proof actually shows that strongly NIP fields are perfect\footnote{Shelah's proof only uses the simple fact that if $\mathrm{char}(K)=p>0$ then either $K$ is perfect or $[K^\times: (K^\times)^p]$ is infinite. See, e.g., \cite[Remark 2.5]{KrVal}}. Shelah  conjectured \cite[Conjecture 5.34]{Sh863} that (interpreting its somewhat vague formulation) strongly NIP fields are real closed, algebraically closed or support a definable non-trivial (henselian)  valuation. Recently, this conjecture was proved\footnote{The existence of a definable valuation is implicit in Johnson's work. See Remark \ref{JohnsonDef}.} by Johnson \cite{JohnDPMin} in the special case of dp-minimal fields (and, independently, assuming the definability of the valuation, henselianity is proved in \cite{JaSiWa2015}).

The two main open problems in the field are: 
\begin{enumerate}
	\item Let $K$ be an infinite (strongly) NIP field that is neither separably closed nor real closed. Does $K$ support a non-trivial definable valuation?
	\item Are all (strongly) NIP fields henselian (i.e., admit some non-trivial henselian valuation) or, at least, t-henselian (i.e., elementarily equivalent in the language of rings, to a henselian field)?
\end{enumerate}


A positive answer to Questions (2) would imply, for example, that strongly\footnote{F. Jahnke and S. Anscombe (private communication) informed us that similar results are obtained for NIP fields.} dependent fields are elementarily equivalent to Hahn fields over well understood base fields \cite[Theorem 3.11]{HaHaJa}: 
\begin{description}
	\item[Equi-characteristic] $\mathbb R((t^\Gamma))$, $\mathbb C((t^\Gamma))$ or $\overline{\mathbb F}_p((t^\Gamma))$. 
	\item[Finite residue field] $Q((t^\Gamma))$ where $Q$ is a p-adically closed field if the field admits a henselian valuation with finite residue field. 
	\item[Kaplansky] $L((t^\Gamma))$ where $L$ is a rank 1 Kaplansky field with residue field as in (1) above. 
\end{description} 
where in all cases $\Gamma$ is a strongly dependent ordered abelian group (see \cite{HalHas} for the classification of such groups).


In view of the above, a natural strategy for studying Shelah's conjecture would be to, on the one hand, study the conjecture for Hahn fields (with dependent residue fields), as the key example and -- on the other hand -- using the information gained in the study of Hahn fields, try to generalise Johnson's results from dp-minimal fields to the strongly dependent setting. 

The simplest extension of Johnson's proof of Shelah's conjecture for dp-minimal fields would be to finite extensions of dp-minimal fields. 
Section \ref{dp-min} is dedicated to showing that this extension is vacuous, namely we prove that 
a finite extension of a dp-minimal field is again dp-minimal (see Theorem \ref{finite}).
The proof builds heavily on Johnson's classification of dp-minimal fields.

Section \ref{Hahn} is dedicated to the study of (strongly) dependent Hahn fields. Collecting known results of the first author (based on unpublished work of Koenigsmann), we show that Hahn fields that are neither algebraically nor real closed support a definable non-trivial valuation with a $t$-henselian topology. We use Hahn fields to provide examples proving that perfection and boundedness -- the conjectural division lines for simple fields -- are not valid in the NIP case. Building on previous results of Delon \cite{DelHenselian}, B\'elair \cite{BelHenselian} and Jahnke-Simon \cite{JaSiTransfer} we construct the following examples (see Theorem \ref{examples}):
	There are NIP fields with the following properties:
	\begin{enumerate}
		\item A strongly NIP field that is not dp-minimal.
		\item A strongly NIP field $K$ such that $[K^\times: (K^\times)^q ]=\infty$ for some prime $q$.
		\item A perfect NIP field that is not strongly NIP.
		\item An unbounded strongly NIP field.
	\end{enumerate}



In the last two sections of the paper we turn to the problem of constructing definable valuations on (strongly) NIP fields. As Johnson's methods of \cite{JohnDPMin} do not seem to generalise easily even to the finite dp-rank case, we study a more general construction due to Koenigsmann. 



We give, provided the field $K$ is neither real closed nor separably closed (without further model theoretic assumptions), an explicit first order sentence $\psi_K$ in the language of rings such that $K\models \psi_K$ implies the existence of a non-trivial valuation ring definable (over the same parameters appearing in $\psi_K$) in the language of rings. As we will show (see the discussion following Corollary \ref{Q2toQ1}) if $K$ is $t$-henselian then $K\models \psi_K$. Thus to provide a positive answer to Question (1) it will suffice to show that any NIP field (infinite not real closed or separably closed) $K\models \psi_K$, which is also a necessary condition for a positive answer to Question (2). 


Implicit in the work of Koenigsmann \cite{Koe2}, a sentence with roughly the same properties as $\psi_K$ above can certainly be extracted from \cite{Du2016}. However, the sentence $\psi_K$ obtained in Proposition \ref{PropSimplifiedVAxiomsnontrivialdefinable} of this paper is simpler in quantifier depth and in length. As a result the strategy proposed for tackling Question (1) above can be summarised as follows:

\begin{con}\label{fo}
	Let $K$ be an infinite field not separably closed. For any prime $q\neq \mathrm{char}(K)$ let $T_q:=(K^\times)^q+1$. Assume that
	\begin{enumerate}
		\item $T_q\neq K\setminus \{1\}$
		\item $\sqrt{-1}\in K$
		\item There exists $\zeta_q\in K$ a primitive $q$-th root of unity.
	\end{enumerate}
	and at least one of the following holds:
	\begin{enumerate}
		\item $K\models (\exists a_1,a_2)(\{0\}=a_1T_q\cap a_2 T_q)$
		\item $K\models (\forall a_1,a_2\exists b)(b\notin T_q\land b\in (a_1T_q\cap a_2T_q)-(a_1T_q\cap a_2T_q))$
		\item $K\models (\forall a_1,a_2\exists b)(b\notin T_q\land b\in (a_1T_q\cap a_2T_q)\cdot (a_1T_q\cap a_2T_q))$
		\item $K\models (\forall a_1,a_2\exists x,y)(xy\in a_1T_q\cap a_2T_q\land x\notin a_1T_q\cap a_2T_q\land y\notin a_1T_q\cap a_2T_q)$
	\end{enumerate}
	Then $K$ has IP.
\end{con}

\noindent\emph{Acknowledgements.} We would like to thank I. Efrat, M. Hils, F. Jahnke, M. Kamensky, F.-V. Kuhlmann and P. Simon for several ideas, corrections and suggestions.

\section{Preliminaries}\label{prelims}

Throughout we use standard valued fields terminology and notation: $K, L, F$ will  be fields, $\cO_K$ will denote a valuation ring on $K$ with maximal ideal $\cM_K$ (we will drop the subscript $K$ if it is clear from the context). Valuations on $K$ will be denoted by $v,w$ and $\cO_v:=\left\{x\in K: v(x)\geq 0\right\}$, $\cM_v$ the valuation ring associated with $v$ and its maximal ideal respectively.  The reader is referred to any standard textbook on the subject (e.g., \cite{EnPr}) for more details. A non-trivial valuation $v$ on a field $K$ induces a Hausdorff field topology (generated by open balls $B_\gamma(x):=\{y:v(x)>\gamma\}$). It is well known that such topologies can be characterised:  
\begin{de}\label{Vtop}
	A collection of subsets $\cN$ of $K$ is a basis of 0-neighbourhoods for a \emph{V-topology} on $K$ if is satisfies the following axioms:
	\begin{description}
		\item[\textbf{\textup{(V\,1)}}]  $\bigcap \mathcal{N}:=\bigcap_{U\in\mathcal{N}}U=\left\{0\right\}$ and $\left\{0\right\}\notin\mathcal{N}$;
		\item[\textbf{\textup{(V\,2)}}]  $\forall\, U,\,V\ \exists\, W\   W\subseteq U\cap V$;
		\item[\textbf{\textup{(V\,3)}}] $\forall\, U\   \exists\,  V\   V-V\subseteq U$;
		\item[\textbf{\textup{(V\,4)}}] $\forall\, U\   \forall\, x,\,y\in K\   \exists\, V\   \left(x+V\right)\cdot\left(y+V\right)\subseteq x\cdot y+U$;
		\item[\textbf{\textup{(V\,5)}}]  $\forall\, U\   \forall\, x\in K^{\times}\  \exists\, V\   \left(x+V\right)^{-1}\subseteq x^{-1}+U$;
		\item[\textbf{\textup{(V\,6)}}]  $\forall\, U\   \exists\, V\   \forall\, x,\,y\in K\   (x\cdot y\in V\to (x\in U\, \vee\,  y\in U))$.
	\end{description}
\end{de}
It is not hard to check that if $(K,v)$ is a valued field (or a field with an absolute value) then the collection of open balls is a V-topology on $K$. More importantly, the converse is also true (see \cite[Appendix B]{EnPr}): any V-topology on a field $K$ arises in this way.

In the present paper we will investigate and apply a standard technique for constructing a V-topology on a field $K$ from a multiplicative sub-group $G\le K^\times$. We will be following a construction due to Koenigsmann, \cite{Koe2}, but the general method is well known (see \cite[\S11]{EfBook} and references therein). Fix $K$ an infinite field and let $G$ be a multiplicative subgroup of $K^\times$ with $G\neq K^\times$.


%

Given a group $G\le K^\times$ we let $\cT_G$ be the coarsest topology for which $G$ is open and linear transformations are continuous. As shown in \cite[Theorem~3.3]{Du2016} $\cS_G:=\left\{a\cdot G+b: a\in K^\times, b\in K\right\}$ is a subbase of $\cT_G$. Hence
\[\cB_G:=\left\{\left.\bigcap_{i=1}^n\left(a_i\cdot G+b_i\right)\,\right|\, n\in \N,\, a_1,\ldots,a_n\in K^\times, b_1,\ldots,b_n\in K\right\}\] is a base for $\cT_G$.

A simple calculation shows that
\[
	\cN_G:=\left.\left\{U\in \cB_G \ \right|\ 0\in U\right\}
 	\:=\left.\left\{\bigcap_{i=1}^n a_i\cdot \left(- G+1\right) \ \right|\ n\in\N,\,  a_i\in K^\times\right\}
\]
is a base of neighbourhoods of zero for $\cT_G$.
If further $-1\in G$ then
\[
	\cN_G=\left\{\bigcap_{i=1}^n a_i\cdot \left(G+1\right): \ n\in\N,\, a_i\in K^\times\right\}.
\]

Throughout the paper $U,\,V$ and $W$, possibly with indices, will always denote elements of $\cN_G$.

 It follows from \cite[Lemma~3.6 and Corollary 3.8]{Du2016} that, if $\cT_G$ is a basis for a V-topology (see Fact \ref{corVAxiomsnontrivialdefinable} below) then already
 \[\left\{\left(a_1\cdot G+b_1\right)\cap\left(a_2\cdot G+b_2\right): a_1,a_2\in K^\times,\, b_1,b_2\in K\right\}\]
 is a base for $\cT_G$. Hence, if $-1\in G$, \[\cN_G':=\left\{\left(a_1\cdot \left(G+1\right)\right)\cap\left(a_2\cdot\left( G+1\right)\right): a_1,a_2\in K^\times\right\}\] is a base of the neighbourhoods of zero for $\cT_G$.
 As in most of the paper it will be more convenient to work with arbitrary intersections, we will mostly choose to work with $\cN_G$. The advantage of the basis $\cN_G'$ is that, if $G$ is definable (as will be the case) it is a definable basis of $0$-neighbourhoods.
 
 The starting point of the present paper is the following result of Koenigsmann\footnote{A valuation $v$ on $K$ is $p$-henselian if it extends uniquely to $K(p)$ the compositum of all Galois extensions of $K$ of degree $p^n$ (any $n$). For the purposes of the present paper the fact that any henselian valuation is $p$-henselian will suffice. For more information see \cite{KoepHens}.}\footnote{As pointed out by the referee, a correct proof Koenigsmann's result can be found in \cite{JahKo}}.,  \cite{KoepHens}: 
 \begin{fact}\label{corVAxiomsnontrivialdefinable}
 	Let $K$ be a field of characteristic $p$ (possibly 0) and $q$ a prime different from $p$. Let $G:=(K^\times)^q\subsetneq K^\times $ and assume that $\zeta_p\in K$ for a primitive $q$th root of unity. Then $K$ is $q$-henselian if and only if $\cT_G$ is a basis for a $V$-topology, if and only if the canonical $p$-henselian valuation, $v_p$ is $\emptyset$-definable (in which case $\cT_G$ is the topology induced by $v_q$). 
 \end{fact}
\begin{proof}
	By \cite[Theorem 2.1]{KoepHens} $K$ is $p$-henselian if and only if $\cT_G$ generates the same topology as $v$ for some $p$-henselian valuation. By \cite[Main theorem]{JahKo} if $K$ is $p$-henselian then the canonical $p$-henselian valuation is definable. The statement concerning the topologies also follows from \cite[Theorem 2.1]{KoepHens}.
\end{proof}

In the above, and throughout, by \emph{definable} we mean \emph{definable in the language $\cL$ of rings} and by saying that a valuation $v$ on $K$ is definable we mean that $\cO_v$ is $\cL(K)$-definable (where $\cL(K)$  is the expansion of the language $\cL$ by constants for all  elements of $K$).

Let us now explain how the above fact will be applied. Let $K$ be an NIP field. We aim to find conditions for the existence of a definable non-trivial valuation on $K$. By \cite[Theorem II.4.11]{Sh1} (\cite[Observation 1.4]{Sh863}) if $T$ is (strongly) NIP then so is $T^{eq}$. Thus any finite extension of $K$ is also (strongly) NIP. It will suffice, therefore,  to find a definable non-trivial valuation on some
finite extension $L\ge K$ (since if $\mathcal O$ is a non-trivial valuation ring on $L$ then $\mathcal O\cap K$ is a non-trivial valuation ring in $K$). It is therefore, harmless to assume that $\sqrt{-1}\in K$. By
\cite[Theorem~4.4]{KaScWa} $K$ is Artin-Schreier closed. So
the same is true of any finite extension $L\ge K$. This implies (e.g., \cite[Lemma 2.4]{KrVal}\footnote{Krupinski's argument assumes that the field is perfect to conclude that it is algebraically closed. Descarding this additional assumption, and restricting to separable extensions the stronger result follows.}) that either $K$ is separably closed, or there exists some finite separable extension $L\ge K$ and $q\neq \mathrm{char}(K)$ such that $(L^\times)^q\neq L^\times$ (in fact, by \cite[Corollary 4.5]{KaScWa} $K$ has no finite separable extensions of degree divisible by $p$). Since $\sqrt{-1}\in L$ it follows that, letting $L(q)$ denote the $q$-closure of $L$, we have $[L(q):L]=\infty$ (\cite[Theorem 4.3.5]{EnPr}).  So extending $L$ a little more, there is no harm assuming that there exists $\zeta_q\in L$, a primitive $q$th root of unity. Thus, at the price of, possibly, losing the $\emptyset$-definability of the resulting valuation (because of the passage to the field $L$), the basic assumptions of Fact \ref{corVAxiomsnontrivialdefinable} are easily met. So that the application of this result reduces to proving that for $L$ and $q$ as above, $\cN_G$ is a $0$-neighbourhood basis for a V-topology on $L$. Thus, we get the following result (see also \cite{Du2016}):

\begin{cor}\label{groups}
	 Let $K$ be an NIP field that is neither separably closed nor real closed. Then there exists a finite separable field extension $L\ge K$ and a prime $q\neq \mathrm{char}(K)$  such that $(L^\times)^q\neq L^\times$ and $\zeta_q\in L$ for some primitive root of unity. If $K$ is $t$-henselian then for any such $L\ge K$ and $q$ the group $G_q(L):=(L^\times)^q$ satisfies conditions (V1)-(V6) of Definition \ref{Vtop}.
\end{cor}
\begin{proof}
	Assume first that $K$ is henselian, witnessed by a valuation $v$. Then, by the above discussion, as $K$ is neither real closed nor algebraically closed, there is some finite separable field extension $L\ge K$ and prime $q$ such that $G_q(L):=(L^\times)^q$ is a proper subgroup of $L^\times$ and $\zeta_q\in L$. Fix any such extension $L$. Since $v$ is henselian, it extends to a henselian valuation on $L$ which by abuse of notation we will also denote $v$. By \cite[Theorem 5.18]{Du2016} there exists a definable valuation $w$ on $L$ inducing the same topology as both $v$ and $\mathcal O_{G_q(L)}$. In particular $\mathcal O_{G_q(L)}$ is non-trivial. So, by the above discussion, $\mathcal N_{G_q(L)}$ is a basis for a $V$-topology, i.e., it satisfies (V1)-(V6), as required.

	In general, let $\mathcal K\succ K$ be $\aleph_1$-saturated. Then $\mathcal K$ is henselian. Let $L\ge K$ be a finite separable extension such that $G_{q}(L)$ is a proper subgroup. By the primitive element theorem there exists $\alpha\in L$ such that $L=K(\alpha)$. Let $\mathcal L:=\mathcal K(\alpha)$. Then $\mathcal L\succ L$ and $G_{q}(\mathcal L)$ is a proper subgroup. By what we have already shown the group $G_q(\mathcal L)$ satisfies conditions (V1)-(V6). So
	$\mathcal N_{G_q(\mathcal L)}'$ is a also a basis for the topology, so it satisfies the corresponding statements (V1)$'$-(V6)$'$. Since those are first order statements without parameters, they are also satisfied by $G_q(L)$, so $G_q(L)$ also satisfies (V1)-(V6) as required.
\end{proof}

We remind also that by a Theorem of Schmidt \cite[Theorem 4.4.1]{EnPr} any two henselian valuations on a non-separably closed field $K$ are dependent (i.e., generate the same $V$-topology). So we get:

\begin{cor}\label{Q2toQ1}
	Let $K$ be an NIP field that is neither real closed, nor separably closed. If $K$ is henselian, then $K$ supports a definable non-trivial valuation. Moreover, there exists a finite separable extension $L\ge K$ and a prime $q$ such that $G_q(L)\cap K$ generates the same $V$ topology as any henselian valuation on $K$.
\end{cor}
\begin{proof}
	There is no harm assuming that $\sqrt{-1}\in K$. 
	As above, if for all finite separable extensions $L$ and all $q\neq \mathrm{char}(K)$ we have $L^q=L$, we get that $K$ is separably closed, contradicting our assumption. So there are $L\ge K$, a finite extension, and $q$ such that $L^q\neq L$. Since $\sqrt{-1}\in K$ we get that $L(\zeta_q)^q\neq L(\zeta_q)$ for $\zeta_q$, a primitive $q$th root of unity. So there is no harm assuming $\zeta_q\in L$. Since $K$ is henselian, so is $L$. By the previous corollary, $G_q(L)$ generates on $L$ the same topology as any henselian valuation on $L$. The corollary follows.
\end{proof}

Let $L\ge K$ be as provided by the previous corollary. Then $L=K(\alpha)$ for some $\alpha$, and let $f(x)$ be its minimal polynomial over $K$. Let $\bar a$ be the coefficients of $f$. Then $L$ is $K$-interpretable over $\bar a$. So we let $\psi_K$ be the sentence (over $\bar a$) stating that $G:=(L^\times)^q$ satisfies axioms (V1)$'$-(V6)$'$ of a $V$-topology. By Fact \ref{corVAxiomsnontrivialdefinable} if $K\models \psi_K$ then $K$ supports a non-trivial $\bar a$-definable valuation. And if $K$ happens to be $t$-henselian (and, therefore, so is $L$) then $L$ is $q$-henselian, Fact \ref{corVAxiomsnontrivialdefinable} implies that $K\models \psi_K$. Therefore, if $K$ is NIP and we assume the conjecture that any infinite NIP field is ($t$)-henselian then $K\models \psi_K$. 


Assuming that, in the above discussion we do not have to pass to the separable extension $L$ (i.e., $K$ itself satisfies assumptions (1)-(3) of Conjecture \ref{fo}) we get that $K\models \psi_K$ implies the existence of a non-trivial $p$-henselian valuation on $K$ (for some $p$ explicit in $\psi_K$). It is therefore natural to ask: 
\begin{qu}
	If $(K,v)$ is NIP and $v$ is $p$-henselian (for some $p\neq \mathrm{char}(K)$). Is $v$ necessarily henselian? Does this follow, at least, from Shelah's conjecture?
\end{qu}

It is worth pointing out that by \cite[Remark 2.3]{KoepHens} there are fields that are $p$-henselian for all primes $p$ but not henselian. For any such field, $K$, the canonical $p$-henselian valuation (any $p$) is definable, inducing the topology $\cT_G$ for $G=(K^\times)^p$, but this definable valuation is not henselian. So in full generality, the definable valuations discussed in this paper need not be henselian. 


Throughout the paper we will be using without further reference the facts that strongly NIP fields are perfect, that NIP fields are Artin-Schreier closed, and that NIP valued fields of characteristic $p>0$ have a $p$-divisible value group (\cite[Proposition 5.4]{KaScWa}).

\section{dp-minimal fields}\label{dp-min}
Dp-minimal fields are classified in the main result of \cite{JohnDPMin}\footnote{Specific references to Johnson's paper below refer to the publicly available version of the paper, \cite{JohnDPMinArc}.}:
\begin{thm}[Johnson]\label{classification}
	A sufficiently saturated field $K$ is dp-minimal if and only if $K$ is perfect and there exists a valuation $v$ on $K$ such that:
	\begin{enumerate}
		\item $v$ is henselian.
		\item $v$ is defectless (i.e., any finite extension of $(L,v)$ over $(K,v)$ is defectless).
		\item The residue field $Kv$ is either algebraically closed of characteristic $p$ or elementarily equivalent to a local field of characteristic $0$.
		\item The valuation group $\Gamma_v$ is almost divisible, i.e., $[\Gamma_v: n\Gamma_v]<\infty$ for all $n$.
		\item If $\mathrm{char}(Kv)=p\neq \mathrm{char}(K)$  then $[-v(p),v(p)]\subseteq p\Gamma_v$.
	\end{enumerate}
\end{thm}

Given a dp-minimal field $K$ that is not strongly minimal, Johnson constructs an (externally definable) topology \cite[\S3]{JohnDPMinArc}, which he then proves to be a V-topology \cite[\S3 , \S 4]{JohnDPMinArc}. Pushing these results further he proceeds to show \cite[Theorem 5.14]{JohnDPMinArc} that $K$ admits a henselian topology (not necessarily definable). From this we immediately get:

\begin{cor}
	Any dp-minimal field is either real closed, algebraically closed or admits a non-trivial definable henselian valuation. In particular, the V-topology constructed by Johnson is definable and coincides with Koenigsmann's topology, $\cT_G(L)\cap K$, for some finite extension $L\ge K$ and some (equivalently, any)  $G:=(L^\times)^p$ such that $G\neq L^\times$.
\end{cor}
\begin{proof}
	Let $K$ be a dp-minimal field that is neither real closed nor algebraically closed. By \cite[Theorem 5.14]{JohnDPMinArc} $K$ is henselian, and therefore so is any finite extension of $K$. Let $L$ be a finite extension of $K$ such that $G_q(L)\neq L^\times$ and $L$ contains a primitve $q$th root of unity. Then by Fact \ref{corVAxiomsnontrivialdefinable} and Corollary \ref{groups} we get that $L$ admits a non-trivial definable valuation. So $K$ admits a non-trivial definable valuation, and by \cite[Theorem 5.14]{JohnDPMinArc} all definable valuations on $K$ are henselian.
	
	Since $K$ is not separably closed it follows that $K$ supports a unique non-trivial t-henselian topology so the V-topology constructed by Johnson coincides with the topology associated with the definable henselian valuation, and is therefore definable.
\end{proof}

\begin{rem}\label{JohnsonDef}
	\begin{enumerate}
		\item 	The above corollary is implicit in Johnson's work. By inspecting his proof of Theorem 1.2 (\cite[\S 6]{JohnDPMinArc}) one sees that unless $K$ is real closed or algebraically closed the valuation ring $\mathcal O_\infty$  appearing in the proof, the intersection of all definable valuation rings on $K$, is non-trivial, implying that $K$ supports a non-trivial definable valuation.
		\item The same result can also be inferred from \cite[\S 7]{JaSiWa2015}. In that paper it is shown that a dp-minimal valued field which is neither real closed nor algebraically closed supports a non-trivial henselian valuation definable already in the pure field structure. By Johnson's Theorem 5.14 we know that $K$ admits a henselian valuation, which is externally definable. Since an expansion of a dp-minimal field by externally definable sets is again dp-minimal, the result follows.
	\end{enumerate}

\end{rem}

We note that the proof of the first part of the above corollary shows that the same results remain true for finite extensions of dp-minimal fields. This follows also from the following, somewhat surprising, corollary of Theorem \ref{classification}:
\begin{thm}\label{finite}
	Let $K$ be a dp-minimal field, $L$ a finite extension of $K$. Then $L$ is dp-minimal.
\end{thm}
\begin{proof}
	Since dp-minimality is an elementary property, we may assume that $K$ is saturated. Indeed, since $L$ is a finite extension of $K$ it is interpretable in $K$, and if $K'\succ K$ is saturated, the field $L'$ interpreted in $K'$ by the same interpretation is a saturated elementary extension of $L$. Thus, it will suffice to show that there exists a valuation $v$ on $L$ satisfying conditions (1)-(5) of Theorem \ref{classification}. Since $K$ is saturated, there is such a valuation on $K$, extending uniquely to $L$. By abuse of notation we will let $v$ denote also this extension.
	
	Conditions (1) and (2) of the theorem are automatic and condition (4) is an immediate consequence of the fundamental inequality (e.g., Theorem 3.3.4\cite{EnPr}). Condition (3) is automatic if $Kv$ is real closed or algebraically closed. So it remains to check that if $Kv$ is elementarily equivalent to a finite extension of $\mathbb Q_p$ then so is $Lv$. This is probably known, but as we could not find a reference, we give the details.
	
	By Krasner's Lemma any finite extension of $\mathbb Q_p$ is of the form $\mathbb Q_p(\delta)$ for some $\delta$ algebraic over $\mathbb Q$ and  $\mathbb Q_p$ has only finitely many extensions of degree $n$ (for any $n$). Denoting $e(n)$ the number of extensions of $\mathbb Q_p$ of degree $n$, there are $P_1(x),\dots, P_{e(n)}(x)\in \mathbb Q$ irreducible such that any finite extension of $\mathbb{Q}_p$ of degree $n$ is generated by a root of one of $P_1(x), \dots, P_{e(n)}(x)$. As this is clearly an elementary property, we get that the same remains true if $F\equiv \mathbb Q_p$. Of course, all of the above remains true if we replace $\mathbb Q_p$ by some finite extension $L\ge \mathbb Q_p$.
	So if $L'\equiv L$ and $F'\ge L'$ is an extension of degree $n$ it must be that $F'=L(\delta)$ for some $\delta$ realising on of $P_1(x),\dots, P_{e(n)}(x)$, implying that $F'$ is elementarily equivalent to $F$, the  algebraic extension of $L$ obtained by realising the same polynomial.
	
	It remains to show that if $(K,v)$ is of mixed characteristic then $[-v(p),v(p)]\subseteq p\Gamma$ where $p=\mathrm{char} Kv$. By \cite[Lemma 6.8]{JohnDPMinArc} and the sentence following it, to verify this condition it suffices to show that $[-v(p),v(p)]$ is infinite. Towards that end it will suffice to show that $[-v(p),v(p)]\cap v(K)$ is infinite. Indeed, by assumption $(K,v)$ satisfies (5) of Theorem \ref{classification}, so $[-v(p),v(p)]\subseteq p\Gamma$. As we are in the mixed characteristic case $v(p)>0$. Since $v(p)\subseteq p\Gamma$, there is some $g_1\in \Gamma$ such that $pg_1=v(p):=g_0$. So $0<g_1<g_0$, and by induction, for all $n$ we can find $0<g_n<g_{n-1}<g_0=v(p)$. This show that $[-v(p), v(p)]\cap v(K)$ is infinite, concluding the proof of the theorem.
\end{proof}

As already mentioned in the beginning of this section, the V-topologies constructed by Johnson and Koenigsmann coincide in the dp-minimal case. However, in order to start Koenigsmann's construction we first need to assure that $G_q(K)\neq K^\times$, and for that we may have to pass to a finite extension. Let us now point out that in the dp-minimal case this is not needed:
\begin{lem}
	Let $K$ be a dp-minimal field that is neither real closed nor algebraically closed. Then $G_q(K)\neq K^\times$ for some $q$.
\end{lem}
\begin{proof}
	Let $v$ be as provided by Theorem \ref{classification}. It will suffice to show that the value group is not divisible. This is clear if the residue field is elementarily equivalent to a finite extension of $\mathbb Q_p$. Indeed, any finite extension $L$ of $\mathbb Q_p$ is henselian with value group isomorphic to $\mathbb Z$, which is not $n$ divisible for any $n>1$. So $G_n(L)\neq L^\times$ for any such $n$. As this is expressible by a first order sentence with no parameters, it remains true in any $L'\equiv L$.
	
	If $Kv\models ACF_0$ or $Kv\models RCF$, the value group cannot be divisible, as then $K$ would be  algebraically closed (resp. real closed). If $Kv\models ACF_p$ then, as $v$ is henselian defectless $(K,v)$ is algebraically maximal, in which case divisibility of the value group would again imply that $K\models ACF$.
\end{proof}

\section{Hahn Series and related constructions}\label{Hahn}
Little is known on the construction of simple fields. The situation is different in the NIP setting where strong transfer principles for henselian valued fields (see, e.g., \cite{JaSiTransfer} and references therein for the strongest such result to date) allow the construction of many examples of NIP fields. In the present section we sharpen some of these results and exploit them to construct various examples.

For the sake of clarity we remind the definition of strong dependence (in the formulation most convenient for our needs. See \cite[\S2]{Sh863} for more details):
\begin{de}
		A theory $T$ is strongly dependent if whenever $I$ is an infinite linear order, $\{a_t\}_{t\in I}$ an indiscernible sequence (of $\alpha$-tuples, some $\alpha$), and $a$ is a singleton there is an equivalence relation $E$ on $I$ with finitely many convex classes such that for $s\in I$ the sequence $\{a_t: t\in s/E\}$ is $a$-indiscernible.
\end{de}

We show:

\begin{thm}\label{examples}
	There are NIP fields with the following properties:
	\begin{enumerate}
		\item A strongly NIP field that is not dp-minimal.
		\item A strongly NIP field $K$ such that $[K^\times: (K^\times)^q ]=\infty$ for some prime $q$.
		\item A perfect NIP field that is not strongly NIP.
		\item An unbounded strongly NIP field.
	\end{enumerate}
\end{thm}

Recall that a field is \emph{bounded}\footnote{In the literature e.g., \cite{PilPoi}, \cite{PiScWa} a slightly stronger condition is used. The restriction to separable extensions seems, however, more natural and even implicitly implied in some applications.} if for all $n\in \mathbb N$ it has finitely many separable extensions of degree $n$. Super-simple fields are bounded,\cite{PilPoi}, and conjecturally, so are all simple fields. As pointed out to us by F. Wagner, it follows, e.g., from \cite[Theorem 5.10]{PoiGroups} that bounded stable fields are separably closed.

For the sake of completeness we give a different proof, essentially, due to Krupinski, with a less stability-theoretic flavour: Let $K$ be a bounded stable field. Since stability implies NIP $K$ and all its finite extensions are, as already mentioned, Artin-Schreier closed. By an easy strengthening of \cite[Lemma 2.4]{KrVal}, it will suffice to show that $K^q=K$ for all prime $q\neq \mathrm{char}(K)$. Boundedness\footnote{In \cite{KrVal} Krupinski introduces the slightly weaker \emph{radical boundedness}, which suffices for the argument.} implies that were this not the case for some $q$ we would have $1<[K^\times, (K^\times)^q]<\infty$. By \cite[Proposition 4.8]{KrSRNIP} this implies that $K$ is unstable (in fact, that the formula $\exists z(x-y=z^q)$ has the order-property).

As we will see in the concluding section of the present paper, boundedness may also have a role to play in the study of the two questions stated in the Introduction. In view of the results of Theorem \ref{examples} it seems natural to look for model theoretic division lines that will separate the bounded NIP fields\footnote{Added in proof: In \cite{HaHaJa} it is shown that Shelah's conjecture implies that a strongly dependent field is bounded if and only if it is dp-minimal.}.

\begin{rem}
		In \cite[Corollary 3.13]{KaSh} it is shown that in a strongly dependent field $K$ for all but finitely many primes $p$ we have $[K^\times: (K^\times)^q]<\infty$. Clause (2) of Theorem \ref{examples} shows that this result is optimal.
\end{rem}

We will use Hahn series to construct the desired examples. The basic facts that we need are:
\begin{fact}\label{delon}
	A henselian valued field $(K,v)$ of equi-characteristic $0$ is (strongly) NIP if and only if the value group and the residue field are (strongly) NIP.  If $(K,v)$ is dp-minimal then so are the residue field and the value group.
\end{fact}

The NIP case of the above fact is due to Delon \cite{DelHenselian} and the strongly NIP case is due to Chernikov \cite{Chernikov}. We get:
\begin{lem}
	Let $k$ be a field of characteristic $0$, $\Gamma$ an ordered abelian group. Then the Hahn series $k((t^\Gamma))$ is NIP as a valued field if and only if $k$ is NIP as a pure field. It is strongly NIP if and only if $k$ and $\Gamma$ are.
\end{lem}
\begin{proof}
	Hahn series are maximally complete, and therefore henselian. So the result follows from the previous fact.
\end{proof}

In order to prove clauses (1) and (3) of Theorem \ref{examples} it will suffice, therefore, to find strongly NIP ordered abelian groups that are not dp-minimal and ones that are not strongly NIP. We start with the latter:

\begin{ex}
	Consider $\Gamma:=\mathbb Z^\mathbb N$ as an abelian group (with respect to pointwise addition) with the lexicographic order. Then $\Gamma$ is NIP but not strongly NIP.
\end{ex}
\begin{proof}
	The group $\Gamma$ is ordered abelian, and therefore NIP by \cite{GurSch}. But $[\Gamma:n\Gamma]=\infty$ for all $n>1$, whence not strongly NIP by \cite[Corollary 3.13]{KaSh}.
\end{proof}

\begin{rem}
	In \cite{Sh863} Shelah considers a closely related example of an ordered abelian group that is not strongly dependent.
\end{rem}

\begin{ex}
	Let $\Gamma:=\mathbb Z^\mathbb N$. If $k$ is an NIP field of characteristic $0$ then $K:=k((t^\Gamma))$ is NIP by the previous lemma. It is not strongly NIP because $\Gamma$ is not strongly NIP. It is unbounded, since by the fundamental inequality it has infinitely many Kummer extensions of any prime degree $q$. Indeed, for any natural number $n$ let $\{a_1,\dots, a_n\}\in \Gamma$ be pairwise non-equivalent modulo $q\Gamma$. Let $c_1,\dots, c_n\in K$ be such that $v(c_i)=a_i$ for all $i$. Let $L\ge K$ be the extension obtained by adjoining $q^{\text{th}}$-roots for all $c_i$. Let $\Delta=v(L)$, where $v$ is identified with its unique extension to $L$. Then $\Gamma\not\subseteq q\Delta$. Otherwise $[q\Delta:q\Gamma]=[q\Delta:\Gamma][\Gamma:q\Gamma]=\infty$, whereas $[\Delta:\Gamma]\le [L:K]$ and $[q\Delta:q\Gamma]\le [\Delta:\Gamma]$. This is a contradiction. Since $n$ was arbitrary, this shows that $K$ has infinitely many Kummer extensions of degree $q$.
\end{ex}

Note that by \cite[Corollary 3.13]{KaSh} and \cite{JaSiWa2015} if $G$ is an ordered abelian group that is strongly dependent and not dp-minimal then there are finitely many primes $q$ such that $[G:qG]=\infty$. So the previous example with $G$ replacing $\Gamma$ will give an example for Theorem \ref{examples}(1), (2) and (5).

The details of the following example can be found in \cite{HalHas}:
\begin{fact}\label{poschar}
	Let $_{(2)}\mathbb Z$ be the localisation of $\mathbb Z$ at $(2)$. Let $B$ be a base for $\mathbb R$ as a vector space over $\mathbb Q$ and let $\langle B \rangle$ be the $\mathbb Z$-module generated by $B$. Let $G:=_{(2)}\mathbb Z\otimes \langle B \rangle$. Viewed as an additive subgroup of $\mathbb R$ the group $G$ is naturally ordered. It is strongly dependent but not dp-minimal.
\end{fact}

In positive characteristic, the situation is slightly different. The basic result is due to B\'elair \cite{BelHenselian}:
\begin{fact}\label{Belair}
	Let $(K,v)$ be an algebraically maximal Kaplansky field of characteristic $p>0$. Then $K$ is NIP as a valued field if and only if the residue field $k$ is NIP as a pure field.
\end{fact}

This generalises to the strongly dependent setting using the following results: 
\begin{fact}[\cite{Sh863}, Claim 1.17]
	Let $T$ be a theory of valued fields in the Denef-Pas language. If $T$ admits elimination of field quantifiers (\cite[Definition 1.14]{Sh863}) then $T$ is strongly dependent if and only if the value group and the residue field are. 
\end{fact}

\begin{fact}[\cite{BelHenselian}, Theorem 4.4]\label{QE}
	Algebraically maximal Kaplansky fields  of equi-characteristic $(p,p)$ admit elimination of field quantifiers in the Denef-Pas language. 
\end{fact}
The combination of the last two facts extends Fact \ref{Belair} to the strongly NIP case in analogy with Fact \ref{delon}:
\begin{cor}
	Let $k$ be an infinite NIP field (of equi-characteristic $(p,p)$) and $\Gamma$ an ordered abelian group. Then $k((t^\Gamma))$ is NIP provided that, if $p=\mathrm{char}(k)>0$, then $\Gamma$ is $p$-divisible. It is strongly NIP if and only if $k$ and $\Gamma$ are.
\end{cor}

\begin{rem}
	Though in \cite{BelHenselian} B\'elair does not claim Fact \ref{QE} for algebraically maximal Kaplansky fields in mixed characteristic his proof seems to work equally well in that setting. A more self contained proof is available in \cite{HalHasQE}. Combined with \cite[Proposition 5.9]{HalHas} we get that for the last sentence in the above corollary to hold (for algebraically closed Kaplansky fields of any characteristics) we do not need the value group and the residue field to be pure. This gives a strongly dependent version of \cite[Theorem 3.3]{JaSiTransfer}. 
\end{rem}

It is natural to ask whether all NIP fields constructed as Hahn series satisfy Shelah's conjecture, namely, whether they all support a definable henselian valuation. It follows immediately from Corollary \ref{groups} that:
\begin{prop}\label{ConjForHahn}
	Let $k$ be an NIP field, $\Gamma$ an ordered abelian group which is $p$-divisible if $\mathrm{char}(k)=p>0$. Then $K:=k((t^\Gamma))$ is either algebraically closed, or real closed or it supports a definable non-trivial valuation.
\end{prop}

This answers Question (1) for Hahn fields. Whether NIP Hahn fields support a definable \emph{henselian} valuation is more delicate. In positive characteristic this follows from  \cite[Corollary 3.18]{JanKo}. Proposition 4.2 of that same paper provides a positive answer (in any characteristic) in case $K=k((t^\Gamma))$ and $\Gamma$ is not divisible. It seems, however, that the general equi-characteristic 0 case remains open. 
In some cases we can be even more precise. E.g., Hong, \cite{Hong} gives conditions on the value group implying the definability of the natural (Krull)  valuation on $k((t^\Gamma))$:
\begin{fact}
	Let $(K, \mathcal O)$ be a henselian field. If the value group contains a convex $p$-regular subgroup that is not $p$-divisible, then $\mathcal O$ is definable in the language of rings.
\end{fact}


In all the examples discussed in the present section, the source of the complexity of the field (unbounded, strongly dependent not dp-minimal etc.) can be traced back to the value group of the natural (power series) valuation. For example, as shown in \cite{JaSiWa2015}, an ordered abelian group $\Gamma$ is dp-minimal if and only if $[\Gamma:p\Gamma]$ is finite for all primes $p$. By Theorem \ref{classification} dp-minimal fields are henselian with dp-minimal value groups. We note that it also follows from the same theorem that dp-minimal fields are bounded. Indeed\footnote{This argument was sugegsted to us by I. Efrat. Any mistake is, of course, solely, ours.}, for any Henselian NIP field $(K,v)$ with $\mathrm{char}(K)=\mathrm{char}(Kv)$ we have that $G_K\cong T \rtimes G_k$ where $G_K, G_k$ are the respective absolute  Galois groups of $K$ and $k=Kv$, and $T$ is the inertia group (see \cite[Theorem 22.1.1]{EfBook} and use the fact that $K$ has no extensions of degree divisible by $\mathrm{char }(K)$). If $K$ is dp-minimal then $T=\prod_{l\in \Omega} \mathbb Z_l^{\dim_{\mathbb F_l}\Gamma/l\Gamma}$ for a certain set of primes $\Omega$. Since $\Gamma:=vK$ is dp-minimal, this implies that $T$ is small. Since $k$ is either real closed, algebraically closed or elementarily equivalent to  a finite extension of $\mathbb Q_p$, also $G_K$ is small. If $(K,v)$ is of mixed characteristic, the exact same argument works if $v$ has no coarsening $w$ of equi-characteristic $0$. Otherwise, decompose $K\xrightarrow{w}Kw\xrightarrow{\bar v} Kv$ and note that $G_{Kw}$ is small by what we have just written, so $K$ is small by our argument for equi-characteristic $0$. 

It seems, therefore, natural to ask whether the complexity of the value group in the above examples can be recovered definably. Can any (model theoretic) complexity of an NIP field be traced back to that of an ordered abelian group:
\begin{qu}
	Let $K$ be a non separably closed NIP field. Does $K$ interpret a dp-minimal field? If $K$ is not strongly dependent (resp. dp-minimal) is $K$ either imperfect or admits an (externally) definable non-trivial henselian valuation with a non strongly-dependent (resp. dp-minimal) value group? 
\end{qu}

\section{The Axioms of V-Topologies for $\cT_G$}\label{AxiomsRevisted}

We are now returning to that construction of V topologies from multiplicative subgroups, as described in Section \ref{prelims}. Throughout this section no model theoretic assumptions are made, unless explicitly stated otherwise.

For ease of reference, we remind the axioms of V topology:
\begin{de}
	A collection of subsets $\cN$ of a field $K$ is a basis of 0-neighbourhoods for a V-topology on $K$ if is satisfies the following axioms:
	\begin{description}
			\item[\textbf{\textup{(V\,1)}}]  $\bigcap \mathcal{N}:=\bigcap_{U\in\mathcal{N}}U=\left\{0\right\}$ and $\left\{0\right\}\notin\mathcal{N}$;
			\item[\textbf{\textup{(V\,2)}}]  $\forall\, U,\,V\ \exists\, W\   W\subseteq U\cap V$;
			\item[\textbf{\textup{(V\,3)}}] $\forall\, U\   \exists\,  V\   V-V\subseteq U$;
			\item[\textbf{\textup{(V\,4)}}] $\forall\, U\   \forall\, x,\,y\in K\   \exists\, V\   \left(x+V\right)\cdot\left(y+V\right)\subseteq x\cdot y+U$;
			\item[\textbf{\textup{(V\,5)}}]  $\forall\, U\   \forall\, x\in K^{\times}\  \exists\, V\   \left(x+V\right)^{-1}\subseteq x^{-1}+U$;
			\item[\textbf{\textup{(V\,6)}}]  $\forall\, U\   \exists\, V\   \forall\, x,\,y\in K\   x\cdot y\in V\longrightarrow x\in U\, \vee\,  y\in U$.
	\end{description}
\end{de}

\noindent\textbf{Notation:} From now on $G$ will denote a multiplicative subgroup of $K^\times$ with $-1\in G$ and $T:= G+1$. We let $\cN_G:=\left\{\bigcap_{i=1}^n a_i\cdot T: a_i\in K^\times\right\}$, as defined in the opening paragraphs of Section \ref{prelims}.

In this setting the first part of \textbf{\textup{(V\,1)}} is automatic, and \textbf{\textup{(V\,2)}} holds by definition:
\begin{lem}\label{lemV1}\label{lemV2}\begin{enumerate}
		\item 	$\bigcap \cN_G=\left\{0\right\}$.
		\item 	$\forall\, U,\,V\ \exists\, W\    W\subseteq U\cap V$.
	\end{enumerate}
\end{lem}
\begin{proof}
	For every $x\in K^\times$ we have $x\not\in x\cdot  T\in \cN_G$. As $-1\in G$ further $0\in x\cdot  T$ for every $x\in K^\times$. Hence $\bigcap \cN_G=\left\{0\right\}$. This proves (1), item (2) holds by the definition of $\cN_G$.
\end{proof}

We will come back to the second part of Axiom~\textbf{\textup{(V\,1)}} later. Axiom~\textbf{\textup{(V\,3)}} is simplified as follows:
\begin{lem}\label{lemV3V3'}
	The following are equivalent
	\begin{description}
		\item[\textbf{\textup{(V\,3)}}] $\forall\, U\   \exists\,  V\   V-V\subseteq U$.
		\item[\textbf{\textup{(V\,3)$'$}}] $  \exists\,  V\    V-V\subseteq  T$.
		\item[\textbf{\textup{(V\,3)}$^*$}] $\forall\, U\   \exists\,  V\   V+V\subseteq U$.
	\end{description}
\end{lem}

\begin{proof}
	The first implication is obvious.
	By \textbf{\textup{(V\,3)$'$}} there exists $ V=\bigcap_{j=1}^n b_j\cdot  T\in \cN_G$ such that $V-V\subseteq  T$.
	Let $ U=\bigcap_{i=1}^m a_i\cdot T\in \cN_G$.
	For all $i\in \{1,\ldots, m\}$, $j\in \{1,\ldots, n\}$.
	Let  $V':=\bigcap_{i=1}^m\bigcap_{j=1}^{n}\left( a_i\cdot b_{j}\cdot T\right)$. Then by direct computation
	\[
	V'-V'
	\subseteq\bigcap_{i=1}^m a_i\cdot\left(V-V\right)
	\subseteq \bigcap_{i=1}^m a_i\cdot T= U.
	\]
	This shows that \textbf{\textup{(V\,3)}} follows from \textbf{\textup{(V\,3)$'$}}. Replacing $V$ with $V\cap (-V)$ (throughout) we may assume that $V=-V$, proving the equivalence with \textbf{\textup{(V\,3)}$^*$}
\end{proof}

%
%
%
%

In order to simplify Axiom~\textbf{\textup{(V\,4)}} we need:
\begin{lem}\label{lemV4xy0case}
	If $ \exists\, V\   V\cdot V\subseteq  T$ then
	$\forall\, U\    \exists\, V\   V\cdot V\subseteq U$.
\end{lem}

\begin{proof}
	Let $U=\bigcap_{i=1}^ma_i\cdot T\in \cN_G$.
	By assumption there exist\\
	$ V=\bigcap_{j=1}^n b_j\cdot T\in \cN_G$ such that $V\cdot V\subseteq  T$.
	
	Let $  V':=\bigcap_{i=1}^{m}\left(\left(\bigcap_{j=1}^n a_i\cdot  b_j\cdot T\right)\cap \left(\bigcap_{j=1}^n b_j\cdot T\right)\right)\in\cN_G$. Then by direct computation
	\[
	V'\cdot V'
	=\bigcap_{i=1}^{m} a_i\cdot\left(V\cdot V\right)
	\subseteq \bigcap_{i=1}^ma_i\cdot T= U.
	\]
	This proves the claim.
\end{proof}

Now we can prove:
\begin{lem}
	The axiom \\
	\textbf{\textup{(V\,4)}}
	$\forall\, U\   \forall\, x,\,y\in K\   \exists\, V\   \left(x+V\right)\cdot\left(y+V\right)\subseteq x\cdot y+U$\\
	is equivalent to the conjunction of\\
	\textbf{\textup{(V\,4)$'$}} $ \exists\, V\   V\cdot V\subseteq  T$ and \\
	\textbf{\textup{(V\,4)$''$}}
	$\forall\, x\in K\   \exists V  \left(x+V\right)\cdot\left(1+V\right)\subseteq x+ T$.
\end{lem}
\begin{proof}
	\textbf{\textup{(V\,4)$'$}} and \textbf{\textup{(V\,4)$''$}} are special cases of \textbf{\textup{(V\,4)}}. So we prove the other implication.
	
	Let $x,y\in K$ and $ U=\bigcap_{i=1}^m a_i\cdot  T\in \cN_G$. The case $x=y=0$ is Lemma~\ref{lemV4xy0case}. So we assume that $y\neq 0$. For every $i\in \left\{1,\ldots, m\right\}$ we define $\widetilde{a}_i:=a_i\cdot y^{-1}$ and $x_i:=x\cdot\widetilde{a}_i^{-1}$. By \textbf{\textup{(V\,4)$''$}} there exists $V_i\in \cN_G$ such that
	\begin{equation}\label{eqV4subseteqxj+ T}
	\left(x_i+V_i\right)\cdot \left(1+V_i\right)\subseteq x_i+ T.
	\end{equation}
	Let $  V:=\bigcap_{i=1}^m\left(\bigcap_{j=1}^{n_i}\widetilde{a}_i\cdot V_i\right)\cap\left( \bigcap_{j=1}^{n_i}\left(y\cdot V_i\right)\right)\in\cN_G$. Then
	\begin{eqnarray*}
		\left(x+V\right)\cdot\left(y+V\right)
		&\subseteq& \bigcap_{i=1}^m\big(\widetilde{a}_i\cdot y\cdot \left(x_i+V_i\right)\cdot\left(1+V_i\right)\big)\\
		&{\subseteq} &\bigcap_{i=1}^m\big(\widetilde{a}_i\cdot y\cdot\left( x_i+ T\right)\big)= x\cdot y+U.
	\end{eqnarray*}
	Where the last inclusion follows from Equation (\ref{eqV4subseteqxj+ T}). This finishes the proof.
\end{proof}

Assuming \textbf{\textup{(V\,3)$'$}} we can simplify further:
\begin{lem}\label{lemV4V3'V4'}
	The axioms \textbf{\textup{(V\,3)$'$}} and \textbf{\textup{(V\,4)$'$}} imply axion \textbf{\textup{(V\,4)}}.
\end{lem}

\begin{proof}
	By the previous lemma it will suffice to prove the lemma for $U=T$ and $y=1$. The case $x=0$ is automatic from the assumptions and Lemma \ref{lemV3V3'}. So assume $x\in K^\times$.
	By Lemma~\ref{lemV3V3'} there exist $V_1,\, V_2$ such that $V_1+V_1\subseteq  T$,  $V_2+V_2\subseteq V_1$. Further by Lemma~\ref{lemV4xy0case} there exists $V_3$ with $V_3\cdot V_3\subseteq V_2$.
	Define $V:= \left(x^{-1}\cdot V_1\right)\cap V_2 \cap V_3\in \cN_G$.
	Let $v,w\in V$.
	\begin{eqnarray*}
		\left(x+v\right)\cdot \left(1+w\right)
		 \in  x+V+x\cdot V+V\cdot V \subseteq x+V_2+x\cdot x^{-1}\cdot V_1+V_3\cdot V_3 \\ \subseteq  x+V_2+V_1+V_2 \subseteq x+V_1+V_1
		\subseteq x+ T.
	\end{eqnarray*}
	Hence $\left(x+V\right)\cdot\left(1+V\right)\subseteq x+ T$, as required.
	
\end{proof}

The axiom \textbf{\textup{(V\,5)}} holds without further assumptions:
\begin{lem}\label{lemV5}
	Let $K$ be a field. Let  $G$ be a multiplicative subgroup of $K$ with $-1\in G$.
	Then \textbf{\textup{(V\,5)}} $\forall\, U\   \forall\, x\in K^{\times}\  \exists\, V\   \left(x+V\right)^{-1}\subseteq x^{-1}+U$ holds.
\end{lem}

\begin{proof}
	We will first show
	\begin{equation}\label{eq:V5 T}
	\forall\, x\in K^{\times}\  \exists\, V\   \left(x+V\right)^{-1}\subseteq x^{-1}+ T.
	\end{equation}
	For $x=-1$ let  $V:= T$.
	We have
	$\left(x+ T\right)^{-1}=\left(-1+G+1\right)^{-1}=G^{-1}
	= G
	=x^{-1}+ T.$
	
	If $x\in K^\times\setminus\left\{-1\right\}$,
	let $b_1=-x^2\cdot \left(1+x\right)^{-1}$, $b_2=-x$ and $V:=b_1\cdot T\cap b_2\cdot T=b_1\cdot\left(G+1\right)\cap b_2\cdot\left(G+1\right)$.
	Let $z\in \left(x+V\right)^{-1}$. Let $g_1,g_2 \in G$ such that $z=\left(x+b_1\cdot g_1 +b_1\right)^{-1}=\left(x+b_2\cdot g_2 +b_2\right)^{-1}$.
	We have
	\begin{equation}\label{eqV5z}
	z=\left(x+b_2\cdot g_2 +b_2\right)^{-1}=\left(x-x\cdot g_2 -x\right)^{-1}=-x^{-1}\cdot g_2^{-1}.
	\end{equation}
	Further we have
	$z^{-1}= x+b_1+b_1\cdot g_1$ and therefore $1- b_1\cdot g_1\cdot z= \left(x+b_1\right)\cdot z$. This implies
	\begin{eqnarray*}
		z&=& \left(1-b_1\cdot g_1\cdot z\right)\cdot \left(x+b_1\right)^{-1}\\
		&=& x^{-1}+1+x\cdot g_1\cdot z\\
		&\stackrel{\textrm{\scriptsize{(\ref{eqV5z})}}}{=}& x^{-1}+1-x\cdot g_1\cdot x^{-1}\cdot g_2^{-1}\\
		&=& x^{-1}+ \left(- g_1\cdot g_2^{-1}\right)+1
		\in  x^{-1}+G+1.
	\end{eqnarray*}
	Hence $\left(x+V\right)^{-1}\subseteq x^{-1}+ T$.
	This proves Equation~(\ref{eq:V5 T}).
	
	Now let $x\in K^\times$ and $U=\bigcap_{i=1}^m a_i\cdot T\in\cN_G$.
	For every $i\in \left\{1,\ldots, m\right\}$ let $x_i:=a_i\cdot x$. By Equation~(\ref{eq:V5 T}) there exists $V_i$ such that
	\begin{equation}\label{eqV5zwei}
	\left(x_i+V_i\right)^{-1}\subseteq x_i^{-1}+ T.
	\end{equation}
	For $  V:=\bigcap_{i=1}^ma_i^{-1}\cdot V_i$
	\begin{eqnarray*}
		\left(x+V\right)^{-1}
		=\bigcap_{i=1}^m a_i\cdot\left(x_i+V_i \right)^{-1}
		\stackrel{\textrm{\scriptsize{(\ref{eqV5zwei})}}}{\subseteq} \bigcap_{i=1}^m a_i\cdot \left( x_i^{-1}+ T\right)
		= x^{-1}+U.
	\end{eqnarray*}
	Therefore \textbf{\textup{(V\,5)}} holds.
\end{proof}

The axiom \textbf{\textup{(V\,6)}} can be reduced as follows:
\begin{lem}\label{lemV6V6'}
	The following are equivalent
	\begin{description}
		\item[\textbf{\textup{(V\,6)}}]  $  \forall\, U\  \exists\,V\  \forall\,x,y\in K\ (x\cdot y\in V\to  x\in U\vee y\in U)$
		\item[\textbf{\textup{(V\,6)$'$}}]  $ \exists\, V\ \forall\,x,y\in K\ (x\cdot y\in V\to x\in  T\vee y\in  T)$.
	\end{description}
\end{lem}

\begin{proof}
	We assume \textbf{\textup{(V\,6)$'$}} and show \textbf{\textup{(V\,6)}}. We will show by induction on $m$,  that for all $a_1,\ldots, a_m\in K^\times$, there exists
	$V\in \cN_G$ such that for all $x,y\in K$, if $x\cdot y\in V$ then $x\in \bigcap_{i=1}^m a_iT$ or $y\in \bigcap_{i=1}^m a_iT$.

	Let $a_1\in K^\times$ and $U:= a_1\cdot T\in \cN_G$. By \textbf{\textup{(V\,6)$'$}} there exists $V$ such that for all $x,y\in K$, if $x\cdot y\in V$ then $x\in  T$ or $y\in  T$. Define $V':={a_1^2}\cdot V\in \cN_G$. For all $x,y\in K$ such that $x\cdot y\in V'$ we have $  x\cdot a_1^{-1}\cdot y\cdot a_1^{-1}\in a_1^{-2} \cdot V'=V$ and therefore   $x\cdot a_1^{-1}\in  T$ or $y\cdot a_1^{-1}\in  T$ and hence $x\in U$ or $y\in U$.

	Now let $a_1,a_2\in K^\times$ and $  U:=\bigcap_{i=1}^2 a_i\cdot T\in \cN_G$.
	By assumption there exists $V$ such that for all $x,y\in K$ if $x\cdot y\in V$ then $x\in  T$ or $y\in  T$.
	Define $   V'=:a_1^2\cdot V\cap a_2^2\cdot V\cap a_1\cdot a_2\cdot V $.
	Let $x,y\in K$ such that $x\cdot y\in V'$. Then $  x\cdot a_1^{-1}\cdot y\cdot a_1^{-1}\in a_1^{-2}\cdot V'\subseteq V$ and therefore as above
	\begin{equation}\label{oneina1}
	x\in a_1\cdot  T\text{ or }y\in a_1\cdot  T.
	\end{equation}
	and
	\begin{equation}\label{oneina2}
	x\in a_2\cdot  T\text{ or }y\in a_2\cdot  T.
	\end{equation}
	If, by way of contradiction, $x\cdot a_1^{-1}\notin  T$ and $y\cdot a_2^{-1}\notin  T$, then  $  x\cdot a_1^{-1}\cdot y\cdot a_2^{-1}\notin V$, implying $  x\cdot y\notin {a_1}\cdot{a_2}\cdot V\supseteq V'$ contradicting the choice of $x$ and $y$.
	Therefore
	\begin{equation}\label{xina1oryina2}
	x\in a_1\cdot T\text{ or }   y\in a_2\cdot  T.
	\end{equation}
	and, similarly,
	\begin{equation}\label{yina1orxina2}
	y\in a_1\cdot T\text{ or }   x\in a_2\cdot  T.
	\end{equation}
	A straightforward verification shows that equations (\ref{oneina1})-(\ref{yina1orxina2}) implies that if $x\cdot y\in V'$ then either $x\in U$ or $y\in U$.
	%
	%
	%
	
	Now let $m\geq 3$. Assume that for all $a_1,\ldots, a_{m-1}$ there exists  $V$ such that for all $x,y\in K$, if $x\cdot y\in V$ then $x\in \bigcap_{i=1}^{m-1}a_i\cdot T$ or $y\in \bigcap_{i=1}^{m-1}a_i\cdot T$.
	Let $a_1,\ldots, a_m\in K^\times$ and  $  U:=\bigcap_{i=1}^m a_i\cdot T\in \cN_G$. By induction hypothesis for every $j\in \left\{1,\ldots, m\right\}$ there exists $V_{\neq j}$ such that for all $x,y\in K$, if $x\cdot y\in V_{\neq j}$ then $\displaystyle x\in \bigcap_{\stackrel{i=1}{i\neq j}}^m a_i\cdot T$ or $\displaystyle y\in \bigcap_{\stackrel{i=1}{i\neq j}}^m a_i\cdot T$.
	Define $  V:=\bigcap_{i=1}^m V_{\neq i}$. Let $x,y\in K^\times$ such that $x\cdot y\in V$.
	If $x\in a_i\cdot  T$ for all $i\in \left\{1,\ldots, m\right\}$ then $x\in U$ and we are done.
	Otherwise let $j\in \left\{1,\ldots, m\right\}$ with $x\notin a_j\cdot  T$. Let $k,\ell\in \left\{1,\ldots, m\right\}\setminus\left\{j\right\}$ with $k\neq \ell$.
	We have $ x\cdot y\in \bigcap_{i=1}^m V_{\neq i}\subseteq  V_{\neq k}$. As $\displaystyle  x\notin a_j\cdot  T\supseteq \bigcap_{\stackrel{i=1}{i\neq k}}^m a_i\cdot  T$ we have $\displaystyle  y\in \bigcap_{\stackrel{i=1}{i\neq k}}^m a_i\cdot  T$.
	Analogous we show  $\displaystyle  y\in \bigcap_{\stackrel{i=1}{i\neq \ell}}^m a_i\cdot T$.
	Therefore  $\displaystyle  y\in \bigcap_{\stackrel{i=1}{i\neq k}}^m a_i\cdot T \cap \bigcap_{\stackrel{i=1}{i\neq \ell}}^m a_i\cdot T=U$.
	
	Hence for all $U$ there exists $V$ such that for all $x,y\in K$, if $x\cdot y\in V$ then $x\in U$ or $y\in U$.
\end{proof}

Summing up all the simplifications of the present section we obtain:

\begin{prop}\label{PropSimplifiedVAxiomsnontrivialdefinable}
	Let $K$ be a field.
	Let $\mathrm{char}\left(K\right)\neq q$ and if $q= 2$ assume $K$ is not euclidean. Assume that for the primitive $q$th-root of unity $\zeta_q\in K$. Let $G:=\left(K^\times\right)^q\neq K^\times$.
	Assume that
	\begin{description}
		\item[\textbf{\textup{(V\,1)$'$}}]   $\left\{0\right\}\notin\mathcal{N}_G$;
		\item[\textbf{\textup{(V\,3)$'$}}]    $  \exists\,  V\    V-V\subseteq  T$
		\item[\textbf{\textup{(V\,4)$'$}}] $  \exists V \  V\cdot V\subseteq   T$
		\item[\textbf{\textup{(V\,6)$'$}}] $ \exists\, V\ \forall\,x,y\in K\ x\cdot y\in V\to x\in  T\vee y\in  T$.
	\end{description}
	Then $K$ admits a non-trivial definable valuation.
\end{prop}

\begin{proof}
	With Lemma~\ref{lemV1}, Lemma~\ref{lemV2}, Lemma~\ref{lemV3V3'}, Lemma~\ref{lemV4V3'V4'}, Lemma~\ref{lemV5} and Lemma~\ref{lemV6V6'}  the result follows directly from 	Corollary~\ref{corVAxiomsnontrivialdefinable}.
\end{proof}

\section{Back to NIP fields}\label{secNIP}
As already explained in the opening sections, our main motivation in the present paper is to study the existence of definable valuations on (strongly) NIP fields. We also hope that such a project may shed some light on the long standing open conjecture that stable fields are separably closed. We have already explained that in the stable case this conjecture can be rather easily settled under the further assumption that the field is bounded. It is therefore natural to ask whether the same assumption can help settle the questions stated in the Introduction. In the present section we show how boundedness gives quite easily Axiom [\textbf{\textup{(V\,1)$'$}}] (stating that  $\left\{0\right\}\notin\mathcal{N}_G$).


%

If $K$ is an infinite NIP field, i.e. a field definable in a monster model satisfying NIP, then by \cite[Corollary~4.2]{KrSRNIP}  there is a definable additively and multiplicatively invariant Keisler measure on $K$.
In the whole section if not stated differently let $K$ be an infinite NIP field and $\mu$ an additively and multiplicatively invariant definable Keisler measure on $K$.

By \cite[Proposition~4.5]{KrSRNIP} for any definable subset $X$ of $K$ with $\mu(X)>0$ and any $a\in K$, we have
\[\mu\left(\left(a+X\right)\cap X\right)=\mu\left(X\right).\tag{$\clubsuit$}\]
%

\begin{lem}\label{propV1ForGgenerel}
	Let $a_1,\ldots, a_m\in K^\times$ and $G\subseteq K^\times$ a multiplicative subgroup with $-1\in G$ and $\mu(G)>0$.
	Then
	$
	\bigcap_{i=1}^m a_i  T\supsetneq\{0\}.
	$
\end{lem}

\begin{proof}
	As $-1\in G$ it follows that $0\in\bigcap_{i=1}^m a_i  T$.
	
	This also implies that
	\[G+a^{-1}=\{s: 1\in a(G+s)\}\tag{*}\]
	for any $a\in K^\times$.
	
	By additivity of the measure $(\clubsuit)$ applied to the left hand side of $(*)$ gives
	\[\mu(\bigcap_{i=1}^m (G+a_i^{-1}\cap G))=\mu(G)>0. \]
	So by the right hand side of $(*)$ we have $t_0\in \bigcap_{i=1}^m \{s\in  G: 1\in a_i(G+s)\}$. So $1\in a_i(G+t_0)$ for all $1\le i \le m$, and as $t_0\in G$ we get $t_0^{-1}\in \bigcap_{i=1}^m a_i(G+1)=\bigcap_{i=1}^m a_iT$.
	%
	%
\end{proof}

%

\begin{cor}\label{thmV3V4V6Existsdefinablenontrivialvaluation}
	Let $K$ be an infinite NIP field with $\sqrt{-1}\in K$. Let $G:=\left(K^\times\right)^q\neq K^\times$ for some $q\neq \ch(K)$ prime with $\zeta_q\in K$. Assume that $[K^\times: G]<\infty$ and that for $T:=G+1$ we have:
	\begin{description}
		\item[\textbf{\textup{(V\,3)$'$}}]    $  \exists\,  V\    V-V\subseteq  T$
		\item[\textbf{\textup{(V\,4)$'$}}] $  \exists V \  V\cdot V\subseteq   T$
		\item[\textbf{\textup{(V\,6)$'$}}] $ \exists\, V\ \forall\,x,y\in K\ x\cdot y\in V\to x\in  T\vee y\in  T$.
	\end{description}
	Then $K$ admits a non-trivial $\emptyset$-definable valuation.
\end{cor}

\begin{proof}
	By additivity and invariance of $\mu$ we get that $\mu(G)=[K^\times:G]^{-1}$. The result now follows directly from  Proposition~\ref{PropSimplifiedVAxiomsnontrivialdefinable} using Lemma~\ref{propV1ForGgenerel}
\end{proof}

As mentioned in Section~1,
\[\cN_G':=\left\{\left(a_1\cdot \left(G+1\right)\right)\cap\left(a_2\cdot\left( G+1\right)\right): a_1,a_2\in K^\times\right\}\]
is a base of the neighbourhoods of zero of $\cT_G$. We obtain the following corollary:

\begin{cor}\label{remNG'}
	Let $K$ be an infinite NIP field with $\sqrt{-1}\in K$. Let $G:=\left(K^\times\right)^q\neq K^\times$ for some $q\neq \ch(K)$ prime with $\zeta_q\in K$. Assume that $[K^\times: G]<\infty$. Then for $T:=G+1$ we have that
	\begin{description}
		\item[\textbf{\textup{(V\,3)$'$}}]    $  \exists\,  V\in \cN_G \ V-V\subseteq  T$
		\item[\textbf{\textup{(V\,4)$'$}}] $  \exists V\in \cN_G \  V\cdot V\subseteq   T$
		\item[\textbf{\textup{(V\,6)$'$}}] $ \exists\, V\in \cN_G \ \forall\,x,y\in K\ x\cdot y\in V\to x\in  T\vee y\in  T$.
	\end{description}
	if and only if
	\begin{description}
		\item[\textbf{\textup{(V\,3)$'_2$}}] $  \exists\,  \widetilde{V}\in \cN_G'\    \widetilde{V}-\widetilde{V}\subseteq  T$
		\item[\textbf{\textup{(V\,4)$'_2$}}] $  \exists \widetilde{V}\in \cN_G' \  \widetilde{V}\cdot \widetilde{V}\subseteq   T$
		\item[\textbf{\textup{(V\,6)$'_2$}}] $ \exists\, \widetilde{V}\in \cN_G'\ \forall\,x,y\in K\ x\cdot y\in \widetilde{V}\to x\in  T\vee y\in  T$.
	\end{description}
\end{cor}

\begin{proof}
	As $\cN_G'\subseteq \cN_G$ it is clear that if \textbf{\textup{(V\,3)$'_2$}}, \textbf{\textup{(V\,4)$'_2$}} and \textbf{\textup{(V\,6)$'_2$}} hold, then so do
	\textbf{\textup{(V\,3)$'$}}, \textbf{\textup{(V\,4)$'$}} and \textbf{\textup{(V\,6)$'$}}.
	
	On the otherhand if  \textbf{\textup{(V\,3)$'$}}, \textbf{\textup{(V\,4)$'$}} and \textbf{\textup{(V\,6)$'$}} hold, then $\cT_G$ is a V-topology and $\cN_G'$ is a 0-neighbourhood basis for $\cT_G$. Therefore, for any $V\in \cN_G$ witnessing \textbf{\textup{(V\,i)}} ($i=3,4,6$), there exists $\widetilde{V}\in \cN_G$ such that $\widetilde{V}\subseteq V$, and as -- for a fixed $V$ -- the axiom \textbf{\textup{(V\,i)}} is universal, it is automatically satisfied by $\tilde V$.
\end{proof}

Note that 	\textbf{\textup{(V\,3)$'_2$}}, 	\textbf{\textup{(V\,4)$'_2$}} and \textbf{\textup{(V\,6)$'_2$}} are first order sentences in the language of rings (appearing explicitly in the statement of Conjecture \ref{fo}). Let us denote their conjunction as $\psi_K$. Thus, if $K$ is a bounded\footnote{As already mentioned, "radically bounded" would suffice.} NIP field such that $K\models \psi_K$ then $K$ supports a definable valuation.

\section{Concluding remarks}
We have shown in Section \ref{prelims} that if $K$ is infinite NIP there exist a sentence $\psi_K$ (possibly with parameters) such that: 
\begin{enumerate}
	\item If $K\models \psi_K$ then $K$ admits a non-trivial definable valuation. 
	\item If $K$ is $t$-henselian then $K\models \psi_K$. 
	\item If $K^q\neq K$ for some prime $q\neq \mathrm{char}(K)$ then $\psi_K$ (and the definable valuation ring) can be taken over $\emptyset$. 
\end{enumerate}
If $K$ is as in (3) above, $\zeta_q\in K$, a primitive root of unity, and $\sqrt{-1}\in K$ the sentence $\psi_K$ is the statement that $\cN_G'$ is a neighbourhood basis for a $V$ topology for $G=(K^\times)^q$. Assuming for simplicity that the predicates $G$, $T:=1+G$ and $aT$ for $a\in K^\times$ are atomic, $\psi_K$ is the conjunction of (V1)-(V6) which is readily checked to be an AEA-sentence. We have shown in Section \ref{AxiomsRevisted} that $\psi_K$ is equivalent to the conjunction of (V1)$'$, (V3)$'$, (V4)$'$ and (V6)$'$ -- which is an EA-sentence. In the last corollary we have shown that if $K$ is bounded NIP then, in fact, (V1)$'$ automatically holds, reducing further the complexity of $\psi_K$. 

If $K$ does not satisfy (3) above (or the additional assumptions on roots of unity)  we have to replace $K$ with a finite extension $L\ge K$ satisfying the necessary assumptions. In that case $\psi_K$ has to be relativised to $L$ -- which, since $L$ is interpretable in $K$ (possibly with parameters) is not a problem, and does not change the complexity of $\psi_K$ -- provided that, as above, the corresponding field operations, the group $G$, the set $T$ and the open sets $aT$ interpreted in $L$ are assumed atomic.

\end{document}